\pdfoutput=1
\documentclass[twocolumn,amsmath,amssymb,aps,secnumarabic,%
    nofootinbib,groupedaddress,letterpaper]{revtex4-1}
\usepackage{times,mathptmx,amsthm,mathdots}
\newtheorem{theorem}{Theorem}[section]

\theoremstyle{remark}

\DeclareMathOperator{\Spec}{Spec}

\newcommand{\AM}{\mathsf{AM}}
\newcommand{\BPP}{\mathsf{BPP}}
\newcommand{\BQP}{\mathsf{BQP}}
\newcommand{\SZK}{\mathsf{SZK}}
\newcommand{\NP}{\mathsf{NP}}
\newcommand{\PH}{\mathsf{PH}}
\newcommand{\coAM}{\mathsf{coAM}}
\newcommand{\coNP}{\mathsf{coNP}}
\newcommand{\shP}{\mathsf{\#P}}
\renewcommand{\P}{\mathsf{P}}
\newcommand{\X}{\mathsf{X}}

\newcommand{\C}{\mathbb{C}}
\newcommand{\F}{\mathbb{F}}

\newcommand{\Z}{\mathbb{Z}}

\newcommand{\SL}{\mathrm{SL}}

\newcommand{\SU}{\mathrm{SU}}
\newcommand{\Hom}{\mathrm{Hom}}
\newcommand{\adj}{\mathrm{adj}}
\newcommand{\res}{\mathrm{res}}

\newcommand{\balpha}{\bar{\alpha}}
\newcommand{\poly}{\mathrm{poly}}
\newcommand{\vf}{\vec{f}}
\newcommand{\vg}{\vec{g}}
\newcommand{\vt}{\vec{t}}
\newcommand{\vx}{\vec{x}}
\newcommand{\vy}{\vec{y}}
\renewcommand{\bar}{\overline}

\renewcommand{\tensor}{\otimes}

\newcommand{\defeq}{\stackrel{\mathrm{def}}{=}}

\newcommand{\ie}{\textit{i.e.}}

\newcommand{\eatline}{\vspace{-\baselineskip}}

\newcommand{\thm}[1]{Theorem~\ref{#1}}
\renewcommand{\sec}[1]{Section~\ref{#1}}

\newcommand{\eq}[2]{\begin{equation}\label{#1}#2\end{equation}}
\newcommand{\qedeq}[1]{\begin{equation}#1 \tag*{\qedhere}\end{equation}}

\begin{document}
\title{Knottedness is in $\NP$, modulo GRH}

\author{Greg Kuperberg}
\email{greg@math.ucdavis.edu}
\thanks{Partly supported by NSF grant DMS CCF-1013079}
\affiliation{Department of Mathematics, University of
    California, Davis, CA 95616}

\begin{abstract} Given a tame knot $K$ presented in the form of a knot
diagram, we show that the problem of determining whether $K$ is knotted is
in the complexity class $\NP$, assuming the generalized Riemann hypothesis
(GRH).  In other words, there exists a polynomial-length certificate that
can be verified in polynomial time to prove that $K$ is non-trivial.  GRH is
not needed to believe the certificate, but only to find a short certificate.
This result complements the result of Hass, Lagarias, and Pippenger that
unknottedness is in $\NP$.  Our proof is a corollary of major results of
others in algebraic geometry and geometric topology.
\end{abstract}
\maketitle

\section{Introduction}

The algorithmic complexity of unknottedness is a long-standing open problem.
In other words, given a knot $K$ described by a knot diagram or by a
triangulation of its complement, is there a fast algorithm to decide whether
$K$ is the unknot?  (The question makes sense for tame knots; all knots in
this article will be tame.)  Haken \cite{Haken:normal} was the first to
show that there is \emph{any} algorithm.  Currently it is not known whether
unknottedness can be decided in polynomial time.  The question can be written
$$\text{Unknottedness} \stackrel{?}{\in} \P,$$
since $\P$ is the class of yes-no functions (or \emph{yes-no questions}
or \emph{decision problems}) on input strings that can be computed in
polynomial time.

Welsh \cite{Welsh:complexity} proposed the study of qualitative rather
than quantitative bounds on the algorithmic complexity of problems in knot
theory (and by extension, in low-dimensional topology).  In response, Hass,
Lagarias, and Pippenger \cite{HLP:complexity} showed that unknottedness
is in the complexity class $\NP$.  This is the class of yes-no questions
for which an answer of yes can be confirmed in polynomial time with the
aid of an auxiliary string called a witness, a proof, or a certificate.
For example, the question of whether an integer $N$ (written in binary)
is composite is trivially in $\NP$, because a certificate can consist of
a factorization $N = AB$ that proves that $N$ is composite.

By definition, the class $\coNP$ is the class of questions whose
negations, with no and yes switched, are in $\NP$.  In other words, a
problem is in $\coNP$ if there is a certificate for no rather than yes.
For example, it is a non-trivial result that primality, the negation of
compositeness, is in $\NP$, equivalently that compositeness is in $\coNP$
\cite{Pratt:prime}.  This result began an encouraging chain of results.
The computational complexity of primality was improved qualitatively
in stages until finally it was established that primality is in $\P$
\cite{Miller:primality,Rabin:primality,AM:elliptic,GK:elliptic,AKS:primes}.
However, there are other problems that are in both $\NP$ and $\coNP$ that
are thought to be hard, for instance, determining whether an integer is
the product of two primes.

\begin{theorem} Let $K \subset S^3$ be a knot described by a knot diagram,
a generalized triangulation, or an incomplete Heegaard diagram.  Then the
assertion that $K$ is knotted is in $\NP$, assuming the generalized Riemann
hypothesis (GRH).
\label{th:main} \end{theorem}

Together with Hass-Lagarias-Pippenger, we can restate the result as
$$\text{Unknottedness} \in \NP \cap \coNP,$$
assuming GRH.  The complexity theory significance of this result is
that unknottedness is not $\NP$-hard, assuming standard conjectures in
both number theory and complexity theory.  (See \sec{s:complex}.)  It is
instead in the class of intermediate problems, such as graph isomorphism
and factoring integers, that either have undiscovered polynomial-time
algorithms or are hard for some other reason.

We clarify the sense in which \thm{th:main} depends on the generalized
Riemann hypothesis.  Since the role of the verifier for a problem in $\NP$
is to evaluate a proof of ``yes", the verifier might need to assume a
conjecture such as GRH to believe the proof.  This is not the case for
our construction in \thm{th:main}.  Our certificates are unconditionally
convincing, and they always exist.  The only role of GRH is to establish
that the certificate has polynomial length.  Instead of assuming all of
GRH, we can assume a much weaker corollary, \thm{th:low}.  The corollary
asserts that for every non-constant univariate integer polynomial $h(x)$,
there is a moderately small prime $p$ such that $h(x)$ has a root in $\Z/p$.

Our proof of \thm{th:main} quickly follows from major results of others.
Kronheimer and Mrowka \cite{KM:dehn} showed that if $K$ is a non-trivial
knot, then there is a non-commutative representation of $$\rho_\C:\pi_1(S^3
\setminus K) \to \SU(2) \subset \SL(2,\C).$$ Then, simply because the
equations for the representation are algebraic, the complex numbers can be
replaced by a finite field $\Z/p$.  Koiran \cite{Koiran:hilbert} showed that
if a polynomial-length set of algebraic equations has a complex solution, and
if GRH is true, then there is a suitable prime $p$ with only polynomially
many digits.  Thus, the certificate is a prime $p$ and a $2 \times 2$
matrix over $\Z/p$ for each generator of the knot group.  The verifier
must check that the generator matrices satisfy the relations of the knot
group; and that they do not all commute, or in the Wirtinger presentation,
that they are not all equal.  This confirms that $K$ cannot be the unknot.

\nocite{DG:apoly}

\subsection{Related work}
\label{s:related}

See \sec{s:complex} for some discussion of the complexity classes
mentioned here.

Koiran \cite[DIMACS~version]{Koiran:hilbert} showed that the question of
whether integer polynomial equations have a complex solution is in $\AM$,
assuming GRH.  (He also remarks that the problem is $\NP$-hard, so it is
$\NP$-complete assuming both GRH and $\NP = \AM$.)  This implies for us that
knottedness is in $\AM$.  Our idea is to use his constructions
rather than his ultimate result to establish that it is in $\NP$.

Aharonov, Jones, and Landau \cite{AJL:approx} gave a polynomial-time
quantum algorithm to approximate the Jones polynomial $J(K,t)$ of a knot
$K$, uniformly for certain values of $t$.  (See also Freedman, Kitaev,
and Wang \cite{FKW:simulation} for an earlier version of this result.)
Given the conjecture that the Jones polynomial distinguishes the unknot,
this result raises the question of whether unknottedness is in $\BQP$.
The error in the Aharonov-Jones-Landau algorithm has an extra exponential
factor that depends on the knot diagram.  For many values of $t$, any
intrinsic approximation of $J(K,t)$ is $\shP$-hard \cite{Kuperberg:jones}.
Whether unknottedness is in $\BQP$ by some other algorithm remains an
interesting open problem.

Hempel \cite{Hempel:residual} showed that the fundamental group of
any compact 3-manifold with geometric pieces (which is now known
to be all compact 3-manifolds) is residually finite.  Broaddus
\cite{Broaddus:noncyclic} refined Hempel's construction with a similar
purpose to ours.

In 2002, Agol \cite{Agol:talk} announced an algorithm in $\NP$, based
on normal surface theory and sutured manifolds, to test a lower bound
for the genus of a knot.  This would imply that knottedness is in $\NP$,
since the unknot is the only knot of genus 0.  As of this writing, the full
argument has not been published, but other experts in 3-manifold topology
believe it \cite{Lackenby:personal}.  Still, our algorithm here
is completely different and it can be expected to have different further
applications.

In 2005, Hara, Tani, and Yamamoto \cite{HTY:unknotting} claimed that
unknottedness is in $\AM \cap \coAM$.  They have since retracted this result
\cite{Hara:personal}.  In 2011, Musick \cite{Musick:trivial} claimed that
unknottedness is in $\P$.   He has since replaced his arXiv preprint with
a withdrawal notice.

\acknowledgments

The author would like to thank Scott Aaronson, Ian Agol, Alex Coward,
Noam Elkies, Joel Hass, Misha Kapovich, Pascal Koiran, Sawn Lawton, Ben
McReynolds, and Felipe Voloch for useful discussions.

\section{Complexity classes}
\label{s:complex}

We review some popular complexity classes, more to interpret \thm{th:main}
and place it in context than to prove it.  See the Complexity Zoo \cite{zoo}
for a survey of complexity classes that we mention here, and many others.

We (informally) defined the classes $\P$ and $\NP$ in the introduction.
If $D$ is a decision problem, we define the class $\P^D$ to be the set
of decision problems $E$ that can be computed in polynomial time with
the aid of an oracle that answers questions in $D$.  Likewise if $\X$ is
a complexity class, we define $\P^\X$ to be the union of $\P^D$ for all
$D \in \X$.  Likewise we can also define $\NP^D$ and $\NP^\X$.  If $\NP
\subseteq \P^D$, then $D$ is called \emph{$\NP$-hard} (in the Cook or
Turing sense); if also $D \in \NP$, then it is called \emph{$\NP$-complete}.

The \emph{polynomial hierarchy} is the union
$$\PH \defeq \bigcup_{n=1}^\infty
    \underbrace{\NP^{\NP^{\iddots^{\NP}}}}_n.$$
An extension of the conjecture that $\P \ne \NP$ asserts that the polynomial
hierarchy does not collapse, \ie, that each term in this union is larger
than the previous one.  As a rough analogy, the non-collapse of $\PH$ is to
$\P \ne \NP$ in complexity theory, as the generalized Riemann hypothesis
is to the Riemann Hypothesis in number theory.  In particular, if $\NP
= \coNP$, then the polynomial hierarchy collapses to its second level.
Brassard \cite{Brassard:note} showed that
$$\NP \cap \coNP = \P^{\NP \cap \coNP}.$$
Thus if
$$D \in \NP \cap \coNP$$
is $\NP$-hard, then $\NP = \coNP$.

A \emph{promise problem} is a decision function $D(x)$ defined only on some
subset of input strings.  (In other words, the set of input strings that
satisfy a promise to the algorithm.)  An algorithm for a promise problem must
compute the value $D(x)$ when $x$ satisfies the promise, and can be either
yes or no arbitrarily otherwise.  In contrast with non-promise problems,
a promise problem in $\NP \cap \coNP$ can be $\NP$-hard.  For instance,
let $D(x)$ be any $\NP$-complete problem, and define $E(x_1,x_2)$ to simply
be $D(x_1)$, given the promise that exactly one of $D(x_1)$ and $D(x_2)$
is yes.  Then it is not hard to show that $D \in \P^E$, so $E$ is
also $\NP$-hard.

Technically speaking, many problems require a syntactical promise, for
example that the input to a computational knot theory algorithm must describe
a valid knot.  If valid input syntax can be checked in polynomial time,
then it is an irrelevant promise, because incorrect input can efficiently be
assigned the value ``no".  So results such as Brassard's theorem still hold.
This is significant for us, because the validity of an encoding of a knot
diagram can be checked in polynomial time.

The class $\BPP$ is the class of questions that can be answered in polynomial
time with random choices, such that the correct answer is computed
with probability at least $\frac23$.  (Equivalently, the probability
of correctness converges to 1 exponentially with repeated trials.)
The class $\BQP$ is similar but harder to define:  It uses algorithms with
access to quantum randomness, or quantum algorithms.  Quantum algorithms
require quantum computers, which do not yet exist.  Nonetheless, $\BQP$
can be taken seriously as a not-necessarily-realistic complexity class,
just like many other complexity classes.  Membership in $\NP \cap \coNP$
is sometimes taken as evidence of membership in $\BQP$, even though the
most popular conjecture is that neither class contains the other.

The class $\AM$ (Arthur-Merlin) is an important variation of $\NP$.
We imagine that the verifier (Arthur) is supplied a certificate by the prover
(Merlin).  In the class $\NP$, Merlin states a proof, then Arthur analyzes
it deterministically and in polynomial time.  In the class $\AM$, Merlin
statistically persuades Arthur by answering randomly chosen questions in a
bounded number of rounds.  It is known that one round is just as powerful as
any bounded number of rounds \cite{BM:am}, and it is known that Arthur might
as well simply show Merlin coin flips \cite{GS:coins}.  Even though $\AM$
looks more powerful than $\NP$, it is a standard conjecture that $\AM = \NP$.
The conjecture is similar to the conjecture $\BPP = \P$; both conjectures
are true if sufficiently good pseudorandom number generators exist.

Another complexity class worth mentioning is $\SZK$ (statistical zero
knowledge).   One way to define this class is as a subclass of $\AM$ in
which Merlin's proof only tells Arthur things that Arthur already knows.
(In other words, Arthur constructs questions to Merlin backwards from their
answers.  Merlin is only told the questions and must derive the correct
answers with high probability.)   Whereas $\AM$ is unlikely to equal $\coAM$,
the class $\SZK$ is closed under complement \cite{Okamoto:relationships}.
It contains problems such as graph (non)isomorphism \cite{GMW:nothing}.

The class $\shP$ is the class of questions whose answers are non-negative
integers rather than yes or no.  By definition, a question in $\shP$
counts the number of certificates that satisfy a polynomial criterion
together with the input, whereas the analogous question in $\NP$ only asks
if there is at least one certificate.  It is known that $\P^{\shP}$,
which captures $\shP$-hardness, contains all of $\PH$ (and $\BQP$ as
well).  So $\shP$-hardness is thought to be qualitatively much harder
than $\NP$-hardness.

\section{Algebra}

In this section and later, we will often use the polynomial-class
bound, that for some function $f(n)$, there exists an $\alpha \ge 0$
such that
$$f(n) = O(n^\alpha).$$
We will write this as 
$$f(n) = \poly(n),$$
and likewise also the multivariate version.

\subsection{Polynomial equations}

Besides recent geometric topology results, our construction rests mainly on
the following theorem of Koiran.  We define the norm $||f||$ of a possibly
multivariate integer polynomial $f$ to be the maximum of $|a|$, where $a$
is a coefficient of $f$.

\begin{theorem}[Koiran {\cite[Thms.~4~\&~7]{Koiran:hilbert}}] Let
$$f_1,f_2,\ldots,f_m \in \Z[x_1,x_2,\ldots,x_n]$$
be non-constant integer polynomials with
$$d = \max_k \deg f_k \qquad r = \max_k ||f_k||.$$
Suppose that the equation $\vf(\vx) = 0$ has a solution in $\C^n$.  Then it
has a solution of the form
$$(x_1,x_2,\ldots,x_n) =  (g_1(\alpha),g_2(\alpha),\ldots,g_n(\alpha)),$$
where
$$g_1, g_2, \ldots, g_n \in \Z[x]$$
are integer polynomials, $\alpha$ is a root of an irreducible
integer polynomial $h \in \Z[x]$, and
\begin{align*}
\log \deg h &= \poly(n,m,\log d,\log r) \\
\log \log ||h|| &= \poly(n,m,\log d,\log r).
\end{align*}
\label{th:koiran} \end{theorem}
\eatline

Actually, Koiran writes
$$x_k = g_k(\alpha)/a,$$
with
$$\log \log a = \poly(n,m,\log d,\log r).$$
This denominator can just be absorbed into $\alpha$ and $h$.

Now the idea is to find a corresponding root $\balpha$ of $h(x)$ in some
prime field $\Z/p$.

\begin{theorem}[Lagarias-Odlyzko \cite{LO:effective}, Weinberger
\cite{Weinberger:factors}] Let $h \in \Z[x]$ be an irreducible polynomial of
degree $D$ and let $R = ||h||$.  Assuming the generalized Riemann hypothesis,
there is a prime $p$ such that $h(x)$ has a root in $\Z/p$ and such that
$$p = \poly(D,\log R).$$
\label{th:low} \end{theorem}
\eatline

Actually, Lagarias-Odlyzko and Weinberger establish a density version of
\thm{th:low}, which is a special case of the effective \v{C}ebotarev
density theorem.  Koiran \cite[Cor.~1]{Koiran:hilbert} states this density
version as follows: Let $\pi(x)$ be the number of primes $p \le x$, let
$\pi_h(x)$ be the number of primes $p \le x$ modulo which $h(x)$ has a root,
and let $\Delta_h$ be the discriminant of $h$.  Then 
\eq{e:pih}{\pi_h(x) \ge \frac{\pi(x) - \log \Delta_h
    - c\sqrt{x}\log(\Delta_h x^D)}{D},}
where $c$ is an \emph{absolute} constant.  To understand this estimate,
we first recall the prime number theorem:
\eq{e:prime}{\pi(x) \approx \frac{x}{\ln x}.}
Then we recall the formula for the discriminant:
$$\Delta_h = \res(h,h') = \det S(h,h').$$
Here $h'$ is the derivative of $h$, $\res(f,g)$ is the resultant of $f$
and $g$, and $S(f,g)$ is the Sylvester matrix whose determinant is the
resultant.  By separately bounding the terms of $\det S(h,h')$, we obtain
\eq{e:disc}{\log \Delta_h = \poly(D,\log R).}
Equations \eqref{e:pih}, \eqref{e:prime}, and \eqref{e:disc} imply the
estimate in \thm{th:low}.  We only need the weaker estimate
$$\log p = \poly(\log D,\log \log R).$$
We can also interpret \thm{th:low} as stated, or this estimate, as a
separate conjecture in number theory that appears to be much weaker than GRH.

\begin{theorem}[Koiran {\cite[Lem.~3]{Koiran:hilbert}}] Assuming the
hypotheses of \thm{th:koiran} and GRH, there is a prime $p$ with
$$\log p = \poly(n,m,\log d,\log r)$$
such that $\vf(\vx) = 0$ has a solution in $(\Z/p)^n$.
\label{th:koiran2} \end{theorem}

Again, we need only one such prime.  Koiran needs an adequate supply of
them, to show that deciding the solvability of $\vf(\vx) = 0$ over $\C$
is in $\AM$.  He establishes a statistical gap between the cases in which
there is or is not a solution.

\begin{proof} \thm{th:koiran} yields
$$\vf(\vg(\alpha)) = 0,$$
where $h(\alpha) = 0$ and $h(x)$ is irreducible.  \thm{th:low} provides
a prime $p$ such that $h(x)$ has a root $\balpha \in \Z/p$.  Then
\begin{align*}
f_k(g_k(\alpha)) = 0 &\iff h(x)|f_k(g_k(x)) \in \Z[x] \\
    & \implies f_k(g_k(\balpha)) = 0 \in \Z/p.
\end{align*}
So the existence of $\balpha$ yields a modular solution of $\vf(\vx) = 0$.
\end{proof}

\subsection{Algebraic groups}

Our use of \thm{th:koiran2} rests on the fact that $\SL(2)$ is an affine
algebraic group defined over the integers $\Z$.  The concept of an affine
algebraic group yields an organized definition of versions of a group,
for example $\SL(n,F)$, for every base field (or ring) $F$ simultaneously,
using integer polynomial equations.  In our case, $\SL(2)$ is the ring
$$\Z(\SL(2)) \defeq \Z[x_1,x_2,x_3,x_4]/(x_1x_4 - x_2x_3 - 1),$$
since the relation in this ring expresses the assertion
$$\det \begin{pmatrix} x_1 & x_2 \\ x_3 & x_4 \end{pmatrix} = 1.$$
With such a definition, we can then define the group $\SL(2,F)$ (in
the traditional sense of a set with a group law) as the set of ring
homomorphisms
$$\SL(2,F) \defeq \Hom(\Z(\SL(2)),F).$$
This means that $\SL(2,F)$ satisfies the same polynomial relations for
every $F$.  Also, the group law of $\SL(2,F)$ and the matrix inverse
are expressed by the same polynomial equations for every $F$.  (Since the
determinant is 1, inverse $M^{-1}$ of a matrix $M$ is the same as ``adjugate"
matrix $\adj(M)$, which has a polynomial expression in the entries of $M$.)

In general, an \emph{affine algebraic group} (over $\Z$) is defined in
reverse as a commutative ring $R$ (also over $\Z$) together with a reverse
multiplication map, a reverse identity, and an inverse map:
$$m:R \to R \tensor R \qquad i:R \to \Z \qquad s:R \to R.$$
These maps must satisfy the axioms of a commutative Hopf algebra.  The Hopf
algebra axioms imply that for every commutative ring $F$, the set $\Hom(R,F)$
is a group.  We imagine that $R$ is the coordinate ring of a space $G$,
one that is morally a group, even when we have not yet chosen $F$.  We then
write $R = \Z(G)$ and we write $G(F) = \Hom(R,F)$.  In the terminology
of algebraic geometry, $G = \Spec R$ is a group object in the category of
affine schemes, and $G(F)$ is the set of $F$-rational points of $G$.

\begin{theorem} Let $G$ be an affine algebraic group over $\Z$ and
assume GRH.  Let $\Gamma$ be a discrete group with a finite presentation
of length $\ell$.  If there is a homomorphism
$$\rho_\C:\Gamma \to G(\C)$$
with non-commutative image, then there is also a homomorphism
$$\rho_p:\Gamma \to G(\Z/p)$$
with non-commutative image, for a prime $p$ such that
$$\log p = \poly(\ell).$$
\label{th:prime} \end{theorem}
\eatline

\begin{proof} The theorem is a direct application of \thm{th:koiran}.
The condition that $\rho_\C$ or $\rho_p$ is a group homomorphism is the
condition that the images of the generators satisfy the defining relations
of $\Gamma$.  The relations of $\Gamma$ become polynomial equations in the
coordinate ring $\Z(G)$ of the algebraic group $G$.  For a fixed choice
of $G$, encoding the group relations as polynomial equations is only more
complicated by a constant factor.

To establish the non-commutativity property, we use the ``Rabinowitsch trick"
from algebraic geometry.  If $\rho_\C$ is non-commutative, then $\Gamma$
has at least one pair of generators $x$ and $y$ whose images do not commute.
We can write the relation $xy \ne yx$ as a set of polynomial expressions
$\vf(\vx,\vy)$, where $\vx$ and $\vy$ are the variables representing $x$
and $y$, such that for at least one value of $k$,
$$f_k(\vx,\vy) \ne 0.$$
Then we can add new variables $\vt$ together with the polynomial relation
\qedeq{\sum_k t_k f_k(\vx,\vy) = 1.}
\end{proof}

\section{The algorithm}

\subsection{Description of input}

The most traditional way to describe a knot $K$ is by a knot diagram in
the plane with crossings.  Another description is a triangulation of $S^3
\setminus K$, either one in which $K$ is an ideal vertex, or one in which the
peripheral torus of $K$ is triangulated.  A knot complement $S^3 \setminus K$
can also be described by an incomplete Heegaard diagram, in which a solid
torus neighborhood of $K$ appears as one of the 0-handles or 3-handles;
the knot complement is then a special case of a \emph{compression body}.
(We also allow Heegaard diagrams with more than the minimum set of circles.)

Since a triangulation can be thickened to a handle decomposition, it
yields a Heegaard diagram.  (If it is a triangulation of $S^3 \setminus
K$ as a manifold with boundary, we can first collapse the boundary torus
to an ideal vertex.)  The Heegaard diagram has a crossing for every flag
consisting of a triangle and an edge of the triangle.  Therefore if the
triangulation has $n$ tetrahedra, the Heegaard diagram has at most $6n$
crossings, which is a constant factor of overhead.

There is also a canonical conversion from a Heegaard diagram of a knot
complement (or any 3-manifold) to a triangulation, or a knot diagram of
$K$ to a triangulation of $S^3 \setminus K$ with an ideal vertex and two
other vertices.  The easiest way to describe either conversion is by a
simple spine.  A \emph{simple spine} is a 2-complex in $S^3 \setminus K$
which is dual to a generalized triangulation.  If $S^3 \setminus K$ has
a Heegaard diagram, then we can take the simple spine to be the union of
the Heegaard surface $H$ and disks $D_j$ attached at all of the attaching
circles of the Heegaard diagram drawn on $H$.  If the Heegaard diagram has
$n$ crossings, and if every complementary region of the diagram is a disk,
this yields a triangulation with $n$ tetrahedra.

If $K$ has a knot diagram drawn on $S^2$, then we can sew an annulus
onto $S^2$ that follows $K$ and makes a torus sleeve for it in $S^3$
\cite[Fig. 8]{Broaddus:noncyclic}.  If $K$ has $n > 0$ crossings, this
yields a triangulation with $4n$ tetrahedra.

The other conversion is from a triangulation or a Heegaard diagram of
a knot complement to a knot diagram.  An efficient conversion is not
possible in this direction.  For example, if $K$ is a knot and $K'$ is
its $(2,1)$ cable, then their Alexander polynomials satisfy the relation
$$\Delta_{K'}(t) = \Delta_K(t^2).$$ If $K_0$ is a knot with a non-trivial
Alexander polynomial and $K_n$ is the result of $(2,1)$-cabling it $n$
times, then the degree of $\Delta_{K_n}(t)$ grows exponentially in $n$,
and therefore so does the crossing number and even the genus of $K_n$.
On the other hand, $S^3 \setminus K_n$ is a torus sum of $S^3 \setminus
K_0$ with $n$ copies of the $(2,1)$ cable inside of a solid torus.  It
therefore has a triangulation in which the number of simplices only grows
linearly in $n$.

The drawback of a triangulation of a knot complement is that it isn't obvious
when the input is admissible, \ie, when it is in fact a triangulation of
a knot complement.  However, it still makes sense as a promise problem.

\subsection{Proof of the main result}

As stated in the introduction, all of the actual topology in our argument
is in the following result.

\begin{theorem}[Kronheimer-Mrowka \cite{KM:dehn}] If $K \subset S^3$ is
a non-trivial knot, then $\pi_1(S^3\setminus K)$ has a non-commutative
representation in $\SU(2)$.
\label{th:km} \end{theorem}

\begin{proof}[Proof of \thm{th:main}] With any of the given descriptions
of $K$, it is easy to write a polynomial-length presentation of $\pi_1(S^3
\setminus K)$.  In the case of a knot diagram, it is the Wirtinger
presentation.  In the case of a Heegaard diagram, we can discard
extra Heegaard circles and then read the presentation from the 
diagram.  The certificate then consists of a prime $p$ and a
non-commutative group homomorphism
$$\rho_p:\pi_1(S^3 \setminus K) \to \SL(2,\Z/p).$$
By \thm{th:km}, there is a non-commutative group homomorphism
$$\rho_\C:\pi_1(S^3 \setminus K) \to \SU(2) \subset \SL(2,\C).$$
By \thm{th:prime}, $p$ and $\rho_p$ are available such that $p$ has at
most polynomially many digits in the length of the description of $K$.
The group homomorphism is described by its values on the generators of
$\pi_1(S^3 \setminus K)$.  The verifier only needs to check that the
relations hold and that the images of the generators do not all commute.
\end{proof}

\section{Bounds and open problems}

One question is whether \thm{th:main} can be established unconditionally,
\ie, without GRH or \thm{th:low}.  As far as we know, proving \thm{th:low}
unconditionally is a difficult open problem in number theory.  The hypotheses
of Theorems \ref{th:koiran} and \ref{th:koiran2} cannot produce every
possible polynomial that satisfies the hypotheses of \thm{th:low}, given
polynomial bounds on $r$ and $d$.  However, we can design the equations
$\vf(\vx) = 0$ as a straight-line program to directly build a wide variety
of choices for $h(x)$.  In other words, we can choose $n-1$ equations in
the inductive form
$$x_{k+1} = f_k(x_1,x_2,\ldots,x_k),$$
so that $x_k$ for $k > 1$ is a polynomial in $x_1$.  Then a
final equation
$$f_n(x_1,x_2,\ldots,x_n) = 0$$
is directly equivalent to $h(x_1) = 0$ for a very complicated polynomial
$h(x)$.  In this case Koiran's construction cannot do better than to
take $x_1 = \alpha$ and use the same $h(x)$.  It seems very possible that
these choices of $h(x)$ are as difficult, in a complexity-theoretic sense,
as all $h(x)$ that satisfy the hypotheses of \thm{th:low}.

We do not know if the constraints on a homomorphism
$$\pi_1(S^3 \setminus K) \to G(\C),$$
for an algebraic group $G$, are as programmable as general algebraic
equations.  If so, it would be evidence that obtaining the type of
certificate described by \thm{th:main} unconditionally is as difficult as
\thm{th:low} itself.  Of course, there could be some completely different
type of certificate for knottedness, or there could be some good algorithm
(possibly a quantum algorithm) to determine knottedness without the aid
of a certificate.

It is easy to show unconditionally that there is a prime modulo which $h(x)$
has a root, but with a much weaker bound than \thm{th:low}.  For instance,
$h(x)$ attains the value $\pm 1$ for at most $2D$ choices of $x$; thus
we can take some other $x$ with $|x| \le D$ and then let $p$ be a prime
divisor of $h(x)$.  (Or, all of our results work just as well when $h$ has
a root in $\F_q$, if $q$ is a prime power that satisfies the same bound.
We can let $p$ be the first prime modulo which $h$ has positive degree, then
define $\F_q$ by an irreducible factor of $h(x)$ over $\Z/p$.)  We still
obtain a non-commutative representation of $\pi_1(S^3 \setminus K)$ into
a group $\SL(2,\Z/p)$, only with a larger prime $p$.  We thus obtain
the following theorem concerning the existence of such a representation,
which in either case improves similar bounds obtained by Broaddus
\cite{Broaddus:noncyclic}.

\begin{theorem} Let $K$ be an non-trivial knot with $\ell$ crossings
(or an ideal triangulation with $\ell$ simplices).   Then the knot
group $\pi_1(S^3 \setminus K)$ has a finite quotient $G$ with
$$|G| = \exp(\poly(\ell))$$
assuming GRH; or 
$$|G| = \exp(\exp(\poly(\ell)))$$
unconditionally.   In either case, $G \cong \SL(2,\Z/p)$ for some prime $p$.
\end{theorem}

Another interesting bound that we can obtain is an effective version of
Mal'cev's theorem \cite{Malcev:faithful}.  Mal'cev's theorem says that if
$\Gamma$ is a finitely generated group, then $\Gamma$ is residually linear if
and only if it is residually finite.  Bou-Rabee \cite{BouRabee:quantifying}
defined a function $F_\Gamma(n)$ that quantifies how quickly the finite
quotients of $\Gamma$ separate elements of $\Gamma$ from the identity.
More precisely, given $g \in \Gamma \setminus \{1\}$, let $F_\Gamma(g)$
be the minimum cardinality of a finite quotient $G$ of $\Gamma$ in which
$g$ is non-trivial.   Assuming that $\Gamma$ has a distinguished set of
generators, let $F_\Gamma(n)$ be the maximum of $F_\Gamma(g)$ over all $g$
with word length at most $n$.  Bou-Rabee and McReynolds \cite{BM:extremal}
established one effective version of Mal'cev's theorem:   If $\Gamma$
is a finitely generated linear group, then $F_\Gamma(n) = \poly(n)$
(non-uniformly in the choice of $\Gamma$).  Using the methods of this
paper, we can obtain a different type of bound.   Let $L_\Gamma(g)$ be
the smallest dimension of a complex linear representation of $\Gamma$
that distinguishes $g$ from the identity, and let $L_\Gamma(n)$ be the
maximum of $L_\Gamma(g)$ over all $g$ of word length at most $n$.  Then:

\begin{theorem} Suppose that $\Gamma$ is a group with a finite presentation
of total length $\ell$.  Then
$$F_\Gamma(n) = \exp(\poly(n,\ell,L_\Gamma(n)))$$
assuming GRH, and
$$F_\Gamma(n) = \exp(\exp(\poly(n,\ell,L_\Gamma(n))))$$
unconditionally, uniformly in $\Gamma$.
\end{theorem}

This theorem can be proved in the same way as \thm{th:prime}.  In particular,
if $\Gamma$ is a fixed, finitely presented group, the function $F_\Gamma(n)$
cannot grow too much faster than $L_\Gamma(n)$.

We conclude with two questions about the possible intermediate complexity
of unknottedness, short of proving that it lies in $\P$ or in $\BPP$.
Is unknottedness in $\SZK$?  Is it in $\BQP$?   The class $\SZK$ lies in $\AM
\cap \coAM$ (since it lies in $\AM$ by definition and is self-complementary),
and thus conjecturally lies in $\NP \cap \coNP$.  The class $\BQP$ is
not thought to lie in $\NP \cap \coNP$.  However, some important decision
problems (for example those related to factoring and discrete logarithms
in number theory) are known to lie in $\BQP \cap \NP \cap \coNP \cap \SZK$.


\providecommand{\bysame}{\leavevmode\hbox to3em{\hrulefill}\thinspace}
\providecommand{\MR}{\relax\ifhmode\unskip\space\fi MR }
\providecommand{\MRhref}[2]{%
  \href{http://www.ams.org/mathscinet-getitem?mr=#1}{#2}
}
\providecommand{\href}[2]{#2}
\providecommand{\eprint}{\begingroup \urlstyle{tt}\Url}

\end{document}